\documentclass{article}

\usepackage{amsmath,amsthm,amssymb, enumitem}

\usepackage[a4paper, total={5.5in, 8in}]{geometry}

\newtheorem{theorem}{Theorem}[section]
\newtheorem{thmletter}{Theorem}

\newtheorem{lemma}[theorem]{Lemma}
\newtheorem{corollary}[theorem]{Corollary}

\theoremstyle{definition}
\newtheorem{definition}[theorem]{Definition}
\newtheorem{remark}[theorem]{Remark}

\def\N{\mathbb N}
\def\R{\mathbb R}
\def\C{\mathbb C}
\def\S{\mathbb S}

\numberwithin{equation}{section}

\author{Alberto Debernardi}
\title{Hankel transforms of general monotone functions}
\date{}
\begin{document}
	\maketitle

	{\small
		\noindent \textbf{Abstract}:
		We show that the Hankel transform of a general monotone function converges uniformly if and only if the limit function is bounded. To this end, we rely on an Abel-Olivier test for real-valued functions. Analogous results for cosine series are derived as well. We also show that our statements do not hold without the general monotonicity assumption in the case of cosine integrals and series.
		\newline \newline
		\textbf{AMS 2010 Primary subject classification}: 42A38. Secondary: 42A20, 44A20.\newline \newline
		This research was partially funded by the CERCA Programme of the Generalitat de Catalunya, Centre de Recerca Matem\`atica, and the MTM2017--87409--P grant.
	}

	\section{Introduction}
	In this paper we consider the Hankel transform of order $\alpha\geq -1/2$,
	\begin{equation}
	\label{EQhankel}
	H_\alpha f(u)=\int_0^\infty t^{2\alpha+1}f(t)j_\alpha(ut)\, dt,
	\end{equation}
	where $j_\alpha$ is the normalized Bessel function of order $\alpha$. We are interested in studying the conditions on $f$ that are equivalent to the uniform convergence and boundedness (as a function of $u$) of \eqref{EQhankel}. By uniform convergence of \eqref{EQhankel}, we mean that the partial integrals
	\begin{equation}
	\label{EQpartialint}
	\int_0^N t^{2\alpha+1}f(t)j_\alpha(ut)\, dt
	\end{equation}
	converge to $H_\alpha f$ uniformly in $u$ as $N\to \infty$.
	
	As is well known, the Hankel transform of order $\alpha=d/2-1$ (with $d\in \N$) appears as the Fourier transform of a radial function defined on $\R^d$. More precisely, for a radial function $F(x)=f_0(|x|)$ defined on $\R^d$, one has (cf. \cite{SWfourier})
	$$
	\widehat{F}(y)  = |\S^{d-1}|\int_{0}^{\infty} t^{d-1} f_0(t) j_{d/2-1}(2\pi |y|t)\, dt,
	$$
	where $|\S^{d-1}|$ denotes the area of the unit sphere $\S^{d-1} =\{x\in \R^d : |x|=1\}$. If $d=1$, we take $|\S^0|=2$, and $H_{-1/2}f_0$ corresponds to the cosine transform of the even function $f_0$. Note that if $F$ is radial, then $\widehat{F}$ is also radial.

	The study of the convergence of Fourier transforms (or series) often requires monotonicity-type assumptions on the functions (or sequences) involved (see \cite{mio,DLT,askhat,DTgmlipschitz,totik,TikJAT} and the references therein). We also refer the reader to the classical paper \cite{chaundy-jolliffe}, where the problem of uniform convergence and boundedness of sine series was first considered for monotone sequences (see also \cite[Ch.~V]{zygmund}). In the present paper the so-called \textit{general monotonicity} (cf. \cite{LTnachr,TikGM}) will play a central role. From now on, all the functions we consider will be locally integrable, locally of bounded variation on $(0,\infty)$, and vanishing at infinity. We denote $\R_+:=[0,+\infty)$. The integral \eqref{EQhankel} as well as the integrals over infinite intervals are understood as improper Riemann integrals.
	\begin{definition}\label{DEFgm}
		Let $f:\R_+\to \C$. We say that $f$ is general monotone, written $f\in GM$, if there exist constants $C,\lambda>1$ such that for every $x>0$,
		$$
		\int_x^{2x}|df(t)|\leq C\int_{x/\lambda}^{\lambda x} \frac{|f(t)|}{t}\, dt.
		$$
	\end{definition}
	We define the general monotone sequences likewise. We say that a complex-valued sequence $\{a_n\}\in GMS$ if there exist constants $C,\lambda>1$ such that for every $n$,
	$$
	\sum_{k=n}^{2n}|a_k-a_{k+1}|\leq C\sum_{k=n/\lambda}^{\lambda n}\frac{|a_k|}{k}.
	$$
	
	We assume, without loss of generality, that the constant $\lambda$ from the definitions of $GM$ and $GMS$ is a natural number of the form $2^\nu$, $\nu\in \N$.
	
	If we denote by $M$ the class of (nonnegative) monotone functions, it is clear that $M\subsetneq GM$. In this paper we are mainly concerned about \textit{real-valued} $GM$ functions. Those have shown to preserve several desired properties of monotone functions, not only when they are nonnegative, but also in the general real-valued case (see \cite{mio,askhat,DTgmlipschitz,totik}). Here we go one step further and prove Abel-Olivier's test for $GM$ functions (see also \cite{LTZ}).
		\begin{theorem}\label{THM-main-functions}
			Let $f\in GM$ be real-valued. If the integral 
			\begin{equation}
			\label{EQintegrals}
			\int_0^\infty f(t)\, dt
			\end{equation}
			converges, then
			\begin{equation*}
			\label{EQfunctinfinity}
			tf(t)\to 0 \qquad \text{as }t\to \infty.
			\end{equation*}
		\end{theorem}
		If $f\geq 0$, then the convergence of \eqref{EQintegrals} is equivalent to $f\in L^1(\R_+)$, and the statement of Theorem~\ref{THM-main-functions} follows easily from known estimates for $GM$ functions (see Lemma~\ref{LEMgmest} in Section~\ref{SEC4}).
		
		Theorem~\ref{THM-main-functions} can be easily improved to an ``if and only if'' statement.
		\begin{corollary}
			\label{CORrealtest-functions}
			Let $f\in GM$ be real-valued and $\nu\in \R\backslash\{0\}$. If $t^\nu f(t)\to 0$ as $t\to 0$, then $\int_0^\infty t^{\nu-1}f(t)\, dt$ converges if and only if
			$$
			t^\nu f(t)\to 0\, \, \text{as } t\to \infty  \qquad \text{and}\qquad \int_0^\infty t^\nu df(t)\, \, \text{converges,}
			$$
			and moreover,
			$$
			\int_0^\infty t^{\nu-1}f(t)\, dt =-\frac{1}{\nu}\int_0^\infty t^\nu df(t).
			$$
		\end{corollary}
	
	The main goal of this paper is to show that if $f\in GM$  the boundedness of the function $H_\alpha f$ is equivalent to the uniform convergence of the partial integrals \eqref{EQpartialint}, which is not true in general. In fact, in Section~\ref{SEC4} we discuss the general case in detail for the cosine transform (as well as for the cosine series), and show the sharpness of our main results with respect to the general monotonicity assumption. 
	
	It is known that if $f\in GM$ and $t^{2\alpha+1}f(t)\in L^1(0,1)$, then the uniform convergence of \eqref{EQpartialint} is equivalent to the convergence of $\int_0^\infty t^{2\alpha+1}f(t)\, dt$  (see \cite{mio}, where such  a statement has been proved under the assumption that  Theorem~\ref{THM-main-functions} is true). Our main result reads as follows.
	\begin{theorem}
		\label{boundednesscos}
		Let $\alpha\geq -1/2$. Let $f\in GM$ be real-valued and such that $t^{2\alpha+1}f(t)\in L^1(0,1)$. The following are equivalent:
		\begin{enumerate}[label=(\roman{*})]
			\item The integral $\int_0^\infty t^{2\alpha+1}f(t)\, dt$ converges.
			\item The partial integrals $\int_0^N t^{2\alpha+1}f(t)j_\alpha(ut)\, dt$ converge uniformly in $u\in \R_+$ as $N\to \infty$.
			\item The function $H_\alpha f(u)=\int_0^\infty t^{2\alpha+1}f(t)j_\alpha(ut)\, dt$ is bounded on $\R_+$.
		\end{enumerate}
		Moreover, in any of those cases, $M_{2\alpha+2}(f):=\sup_{t\in \R_+}t^{2\alpha+2} |f(t)|$ is finite and for every $N\in \R_+$, the estimate
		\begin{align}
		\bigg| \int_0^N t^{2\alpha+1}f(t)j_\alpha(ut)\, dt\bigg|&\leq \bigg|\int_0^N t^{2\alpha+1}f(t)\, dt\bigg|+\frac{1}{\alpha+1}N^{2\alpha+2}|f(N)| \nonumber\\
		&\phantom{=} +\frac{C\lambda(2\lambda)^{2\alpha+2}}{2\alpha+2}\bigg(\frac{\lambda^4}{2(\alpha+2)}+\frac{S_\alpha}{\alpha+3/2}\bigg)M_{2\alpha+2}(f)\nonumber \\
		&\phantom{=} +\sup_{0\leq a<b\leq \infty}\bigg|\int_a^b \frac{t^{2\alpha+2}}{2\alpha+2}df(t)\bigg| \label{EQ-cossup}
		\end{align}
		holds, where $S_\alpha:=\sup_{x\geq 1}x^{\alpha+1/2}|j_\alpha(x)|$.
	\end{theorem}
	Note that for any $0\leq a<b$, the boundedness of the integral $\int_a^b t^{2\alpha+2} df(t)$ in \eqref{EQ-cossup} follows from the convergence of $\int_0^\infty t^{2\alpha+1}f(t)\, dt$ and Corollary~\ref{CORrealtest-functions}.
	
	It is obvious from the estimates for the Bessel function given in Section~\ref{SECbessel} that the conclusion Theorem~\ref{boundednesscos}  holds if $f\geq 0$.
	
	As an important example we mention the cosine transform, which is the Hankel transform of order $\alpha=-1/2$. For the sake of completeness, we give its corresponding version of Theorem~\ref{boundednesscos}.
	\begin{corollary}\label{CORcostrans}
		Let $f\in GM$ be real-valued and such that $f\in L^1(0,1)$. 
		The following are equivalent:
		\begin{enumerate}[label=(\roman{*})]
			\item The integral $\int_0^\infty f(t)\, dt$ converges.
			\item The partial integrals $\int_0^N f(t)\cos ut\, dt$ converge uniformly in $u\in \R_+$ as $N\to \infty$.
			\item The function $\widehat{f}_{\cos}(u)=\int_0^\infty f(t)\cos ut\, dt$ is bounded on $\R_+$.
		\end{enumerate}
		Moreover, in such case, $\sup_{t\in \R_+}t|f(t)|$ is finite and the estimate
		\begin{align*}
			\bigg| \int_0^N f(t)\cos ut\, dt\bigg|&\leq \bigg|\int_0^N f(t)\, dt\bigg|+2N|f(N)|+2C\lambda^2\bigg(\frac{\lambda^4}{3}+1\bigg)\sup_{t\in \R_+}t|f(t)|\nonumber \\
			&\phantom{=} +\sup_{0\leq a<b\leq \infty}\bigg|\int_a^b t \,df(t)\bigg|
		\end{align*}
		holds.
	\end{corollary}
	
	We are also interested in the discrete part of Corollary~\ref{CORcostrans}. In \cite{TikJAT}, Tikhonov mentioned the following.
	\begin{thmletter}\label{THMuccosine}
		Let $\{a_n\}\in GMS$ be such that $n|a_n|\to 0$. Then the cosine series
		\begin{equation}
		\label{EQcosseries}
		\sum_{n=0}^\infty a_n\cos nx
		\end{equation}
		converges uniformly on $[0,2\pi]$ if and only if the series
		\begin{equation}
		\label{EQseries}
		\sum_{n=0}^\infty a_n
		\end{equation} 
		converges.
	\end{thmletter}
		
	By means of a discrete version of Theorem~\ref{THM-main-functions} we are able to  improve the statement of Theorem~\ref{THMuccosine} by dropping the hypothesis $n|a_n|\to 0$, and eventually allows us to obtain equivalent conditions similarly as in Theorem~\ref{boundednesscos}.
	\begin{theorem}\label{CORunifconv}
		Let $\{a_n\}\in GM$ be real-valued. The following are equivalent.
		\begin{enumerate}[label=(\roman{*})]
			\item The series $\sum_{n=0}^\infty a_n$ converges.
			\item The series $\sum_{n=0}^N a_n\cos nx$ converges uniformly as $N\to \infty$.
			\item The function $\sum_{n=0}^\infty a_n\cos nx$ is bounded.
		\end{enumerate}
	\end{theorem}
	Observe Theorem~\ref{CORunifconv} is trivial if $a_n\geq 0$.
		
	The paper is organized as follows. In Section~\ref{SECbessel}, we state some of the basic properties of the Bessel functions $j_\alpha$ that we use later. We also prove upper and lower estimates for $j_\alpha(t)$, valid for small $t$. In Section~\ref{SEC3}, we prepare the machinery needed to prove Theorem~\ref{boundednesscos}. In particular, Theorem~\ref{THM-main-functions} and its analogue for $GMS$ are proved, as well as Corollary~\ref{CORrealtest-functions}. Finally, in Section~\ref{SEC4} we  prove Theorems~\ref{boundednesscos}~and~\ref{CORunifconv}, and describe which implications do and do not hold in the general case.

	\section{Bessel functions}\label{SECbessel}
	We first present some known properties of the Bessel functions (which can be found in \cite[Ch. VII]{EMOT}), as well as an auxiliary lemma that will be useful later. Throughout this section we let $x\in \R_+$. For $\alpha\geq -1/2$, the \textit{normalized Bessel function of order $\alpha$} is defined as the power series 
	\begin{equation}
	\label{EQ-besselseries}
	j_\alpha(x)=\Gamma(\alpha+1)\sum_{n=0}^\infty \frac{(-1)^n (x/2)^{2n}}{n!\Gamma(n+\alpha+1)},
	\end{equation}
	where $\Gamma$ denotes the Euler gamma function. Such series converges uniformly and absolutely on any bounded interval.	We recall that $j_{-1/2}(x)=\cos x$.
	
	The following property concerning the derivatives of $j_\alpha$ is satisfied:
	\begin{equation}
	\label{EQ-derivative-bessel}
	\frac{d}{dx}\Big(x^{2\alpha+2}j_{\alpha+1}(x)\Big)=(2\alpha+2)x^{2\alpha+1}j_\alpha(x).
	\end{equation}
	For any $\alpha\geq -1/2$ and any $x\geq 0$, one has
	\begin{equation*}
	\label{EQbesselatzero}
	|j_\alpha(x)|\leq j_\alpha(0)=1.
	\end{equation*}
	Finally, we have the following estimate for $x\geq 1$:
	\begin{equation}
	\label{EQ-besselestinfity}
	|j_{\alpha}(x)|\leq S_\alpha x^{-\alpha-1/2}.
	\end{equation}
	\begin{remark}
		We refer the reader to \cite{olenko}, where sharp upper bounds for $S_\alpha$ are obtained. It is known that for $\alpha>1/2$, $S_\alpha$ is strictly increasing to infinity as a function of $\alpha$ and the supremum is attained at the first maximum of the function $x^\alpha j_\alpha(x)$ (see \cite{landau}). Also, it is shown in \cite{olenko} that
		$$\lim_{\alpha\to \infty} \frac{S_\alpha}{\alpha^{1/6}2^\alpha\Gamma(\alpha+1)}=0.6748\ldots$$
	\end{remark}
		
	We now prove upper and lower estimates for the Bessel function near the origin, based on its expansion as power series.
	\begin{lemma}
		\label{LEM-Besselest}
		Let $\alpha\geq -1/2$. For every $x\leq 2\sqrt{\alpha+1}$ and every $m\in \N\cup \{0\}$ there holds
		$$
		\Gamma(\alpha+1)\sum_{n=0}^{2m+1} \frac{(-1)^n (x/2)^{2n}}{n!\Gamma(n+\alpha+1)} \leq j_\alpha(x)\leq \Gamma(\alpha+1)\sum_{n=0}^{2m} \frac{(-1)^n (x/2)^{2n}}{n!\Gamma(n+\alpha+1)}.
		$$
	\end{lemma}
	\begin{proof}
		The proof relies on the fact that for every alternating series $\sum_{n=0}^\infty (-1)^n a_n$ with terms $a_n$ decreasing to zero, the estimate
		$$
		\sum_{n=0}^{2m+1}(-1)^n a_n\leq \sum_{n=0}^\infty (-1)^n a_n\leq \sum_{n=0}^{2m} (-1)^n a_n
		$$
		holds for every $m\in \N\cup \{0\}$. Thus, the result follows if we just prove that for any fixed $x\leq 2\sqrt{\alpha+1}$, the terms of the series \eqref{EQ-besselseries} are decreasing to zero (in absolute value). That is equivalent to say
		$$
		\frac{(x/2)^{2n}}{n! \Gamma(n+\alpha+1)}-\frac{(x/2)^{2n+2}}{(n+1)! \Gamma(n+\alpha+2)}\geq 0.
		$$
		Routine simplifications show that the above inequality is equivalent to 
		$$
		x\leq 2\sqrt{(n+1)(n+\alpha+1)},
		$$
		and the latter holds for every $n\in \N\cup\{0\}$ if and only if it holds for $n=0$, i.e., if and only if $x\leq 2\sqrt{\alpha+1}$.
	\end{proof}
	\section{Abel-Olivier test for $GM$ functions and sequences}\label{SEC3}
	In order to prove Theorem~\ref{THM-main-functions} we adapt two lemmas  obtained by Dyachenko and Tikhonov in \cite{DTgmlipschitz} to the framework of $GM$ functions. Let us define, for any function $f$ and any  $n\in \N$,
	\begin{align*}
	A_n&:=\sup_{2^n\leq t\leq 2^{n+1}}|f(t)|,\\
	B_n&:=\sup_{2^{n-2\nu}\leq t\leq 2^{n+2\nu}}|f(t)|.
	\end{align*}
	
	For $n\in \N\cup \{0\}$, we say that $n$ is a \textit{good} number if either $n=0$ or $B_{n}\leq 2^{4\nu}A_{n}$. The rest of natural numbers consists of \textit{bad} numbers. Recall the parameter $\nu$ comes from the $GM$ condition.
	
	To illustrate such definitions we give a couple of examples. On the one hand, if $f(t)=1/t^2$ for $t\geq 1$, since
	$$
	A_n=\frac{1}{2^{2n}}, \qquad B_n= \frac{1}{2^{2n-4\nu}},
	$$
	then $B_n=2^{4\nu}A_n$, and all natural numbers $n$ (associated to $f$) are good. On the other hand, if $f(t)=1/t^3$ for $t\geq 1$, since
	$$
	A_n=\frac{1}{2^{3n}}, \qquad B_n= \frac{1}{2^{3n-6\nu}},
	$$
	then $B_n=2^{6\nu}A_n\not\leq 2^{4\nu}A_n$, thus all natural numbers $n$ are bad. More generally, if $f$ decreases rapidly enough (faster than $1/t^2$, as for instance $1/t^3$ or $e^{-t}$), then all numbers $n\neq 0$ associated to $f$ are bad.
	\begin{lemma}
		\label{LEMgood1-functions}
		Let $f$ be a $GM$ function. For any good number $n> 0$, there holds
		\begin{equation}
		\label{EQgoodlemma1-functions}
		|E_n|:=\bigg| \bigg\{ x\in[2^{n-\nu},2^{n+\nu}]: |f(x)|>\frac{A_n}{8C2^{2\nu}}\bigg\}\bigg|\geq   \frac{2^n}{8C 2^{5\nu}},
		\end{equation}
		where $|E|$ denotes the Lebesgue measure of $E$ and $C$ is the constant from the $GM$ condition.
	\end{lemma}
	\begin{proof}
		The proof just consists on rewriting that of \cite[Lemma~2.1]{DTgmlipschitz} (dealing with sequences) in the context of functions. Assume \eqref{EQgoodlemma1-functions}	does not hold for $n> 0$. Let us define $D_n:=[2^{n-\nu},2^{n+\nu}]\backslash E_n$. Then, since $n$ is good,
		\begin{align*}
		\int_{2^{n-\nu}}^{2^{n+\nu}} \frac{|f(x)|}{x}\, dx& =\int_{D_n} \frac{|f(x)|}{x}\, dx + \int_{E_n} \frac{|f(x)|}{x}\, dx\\
		&\leq \frac{2^{n+\nu}A_n}{8C2^{2\nu}2^{n-\nu}}+\frac{2^n B_n}{8C2^{5\nu}2^{n-\nu}}= \frac{B_n}{8C 2^{4\nu}}+\frac{A_n}{8C}\leq \frac{A_n}{4C}.
		\end{align*}
		The $GM$ condition implies that for any $x\in [2^n,2^{n+1}]$, 
		$$
		|f(x)|\geq A_n-\int_{2^n}^{2^{n+1}}|df(t)|\geq A_n -C\int_{2^{n-\nu}}^{2^{n+\nu}} \frac{|f(t)|}{t}\, dt \geq A_n - \frac{A_n}{4}>\frac{A_n}{2},
		$$
		which contradicts our assumption.
	\end{proof}
	Before stating the next lemma, let us introduce the following notation:
	$$
	E_n^+:=\{ x\in E_n : f(x)>0\}, \qquad E_n^-:=\{ x\in E_n: f(x)\leq 0\}.
	$$
	\begin{lemma}
		\label{LEMgood2-functions}
		Let $f\in GM$ be real-valued. For any good number $n> 0$ there is an interval $(\ell_n,m_n)\subset [2^{n-\nu},2^{n+\nu}]$ such that at least one of the following holds:
		\begin{enumerate}
			\item for any $x\in (\ell_n,m_n)$, there holds $f(x)\geq 0$ and
			$$
			|E_n^+ \cap (\ell_n,m_n)|\geq \frac{2^n}{256C^3 2^{15\nu}};
			$$
			\item for any  $x\in (\ell_n,m_n)$, there holds $f(x)\leq 0$ and
			$$
			|E_n^- \cap (\ell_n,m_n)|\geq \frac{2^n}{256C^3 2^{15\nu}},
			$$
		\end{enumerate}
		where $C$ is the constant from the $GM$ condition.
	\end{lemma}
	\begin{proof}
		First of all, note that by Lemma~\ref{LEMgood1-functions} one has that either $|E_n^+|\geq \dfrac{2^n}{16C2^{5\nu}}$ or $|E_n^-|\geq \dfrac{2^n}{16C2^{5\nu}}$. We assume the former, and prove that item 1. holds.
		
		Let us construct a system of disjoint intervals $\{I_j=[s_j,t_j]\}_{j=1}^{p_n}$ in $\big[ 2^{n-\nu},2^{n+\nu}+\frac{2^n}{256C^3 2^{15\nu}}\big]$, as follows: Let $s_1=\inf E_n^+$, and
		$$
		\tau_1=\inf\{x\in [s_1,2^{n+\nu}]: f(x)\leq 0\}.
		$$
		If such $\tau_1$ does not exist, then we simply let $t_1=2^{n+\nu}$ and finish the process.	Contrarily, we define
		$$
		t_1=\tau_1+\frac{2^n}{256C^32^{15\nu}}.
		$$
		Once we have the first interval $I_1=[s_1,t_1]$, if $|E_n^+\backslash I_1|>0$, we let $s_2=\inf E_n^+\backslash I_1$, and define $\tau_2$ similarly as above, thus obtaining a new interval $I_2=[s_2,t_2]$. We continue this process until our collection of intervals is such that
		$$
		|E_n^+\backslash (I_1\cup I_2\cup \cdots \cup I_{p_n})|=0.
		$$ 
				
		By construction, for any $1\leq j\leq p_n-1$, we can find $y_j\in [s_j,\tau_j]$ such that $y_j\in E_n^+$, and $z_j\in [\tau_j,t_j]$ such that $f(z_j)\leq 0$. Thus,
		$$
		\int_{I_j}|df(t)|=\int_{s_j}^{t_j}|df(t)|\geq f(y_j)-f(z_j)\geq f(y_j)>\frac{A_n}{8C2^{2\nu}}.
		$$
		Hence,
		$$
		\int_{2^{n-\nu}}^{2^{n+\nu}}|df(t)|\geq \sum_{j=1}^{p_n-1}\int_{I_j}|df(t)|\geq (p_n-1)\frac{A_n}{8C2^{2\nu}}.
		$$
		On the other hand, the $GM$ property and the fact that $n$ is good imply that
		\begin{align*}
		\int_{2^{n-\nu}}^{2^{n+\nu}} |df(t)|&\leq C2\nu \int_{2^{n-2\nu}}^{2^{n+2\nu}} \frac{|f(x)|}{x}\, dx \leq C2\nu B_n\int_{2^{n-2\nu}}^{2^{n+2\nu}} \frac{1}{x}\, dx\\
		&= C2\nu B_n \log 2^{4\nu} \leq C2^{4\nu}8\nu^2 A_n\log 2  \leq C 2^{7\nu}A_n.
		\end{align*}		
		We can deduce from the above estimates that
		$$
		p_n \leq 8C^22^{9\nu} +1 \leq 8C^2 2^{10\nu}.
		$$
		By the pigeonhole principle (or Dirichlet's box principle), there is an integer $j$ such that 
		$$
		|E_n^+ \cap I_j|\geq \frac{2^n}{128 C^3 2^{15\nu}}.
		$$
		Given this $j$, we set $(\ell_n,m_n)=(s_j,t_j-\frac{2^n}{256C^3 2^{15\nu}})=(s_j,\tau_j)\subset [2^{n-\nu},2^{n+\nu}]$, and we are done.
	\end{proof}
	We are in a position to prove Theorem~\ref{THM-main-functions}.
		\begin{proof}[Proof of Theorem~\ref{THM-main-functions}]
			Recall that the convergence of \eqref{EQintegrals}  is equivalent to
			$$
			\bigg|\int_M^N f(t)\, dt\bigg| \to 0\qquad \text{as }N>M\to \infty.
			$$
			We distinguish two cases, namely if there are finitely or infinitely many good numbers. Assume first there are infinitely many. 
			
			For any good number $n>0$, it follows from Lemma~\ref{LEMgood2-functions} that
			$$
			2^nA_n\frac{1}{2048 C^4 2^{17\nu}}<\bigg| \int_{\ell_n}^{m_n}f(t)\, dt\bigg|,
			$$
			and moreover $f(x)\geq 0$ (or $f(x)\leq 0$) for all  $x\in (\ell_n,m_n)\subset [2^{n-\nu},2^{n+\nu}]$. Since the integrals $\int_{\ell_n}^{m_n} f(t)\, dt$ vanish as $n\to\infty$ (by the convergence of \eqref{EQintegrals}) we deduce that
			\begin{equation}
			\label{EQ-goodinfty}
			2^n A_n\to 0 \qquad \text{as }n\to \infty, \, n\text{ good}.
			\end{equation}
			
			We now prove that $2^n A_n$ also vanishes as $n\to \infty$ whenever $n$ is bad. If $n$ is a bad number, then $A_n<2^{-4\nu}B_n$, and $B_n=A_{s_1}$, with $|n-s_1|\leq 2\nu$. Let us first suppose that $s_1<n$ and find the largest good number $m$ which is smaller than $n$. If there is any good number in the set $\{s_1,s_1+1,\ldots , n-1\}$, we just choose $m$ to be the largest good number from such a set and conclude the procedure. On the contrary, $s_1$ is a bad number. Then there exists $s_2$ satisfying $|s_1-s_2|\leq 2\nu$ such that $A_{s_1}<2^{-4\nu}B_{s_1}=2^{-4\nu}A_{s_2}$. Also, note that $s_2<s_1$; the opposite is not possible, as it would imply $A_{s_1}<B_n$, which is a contradiction. Similarly as before, if there is any good number in $\{s_2,s_2+1,\ldots , s_1-1\}$, we choose $m$ to be the largest good number from such a set and we are done.
			
			Repeating this process, we arrive at a finite sequence $n=s_0>s_1>\cdots >s_{j-1}>s_j$, $j\geq 1$, where all the numbers in the set $\{s_{j-1},s_{j-1}+1,\ldots , s_0\}$ are bad, and there exists a good number $m$ satisfying $s_j\leq m<s_{j-1}$ (fix it to be the largest from $\{s_j,s_{j}+1,\ldots ,s_{j-1}-1\}$). Note that $A_{s_k}<2^{-4\nu}A_{s_{k+1}}$ and $|s_k-s_{k+1}|\leq 2\nu$ for any $0\leq k\leq j-1$, thus $n\leq s_j+2j\nu$. Also, the number $m$ obtained by this procedure tends to infinity whenever $n\to \infty$, since there are infinitely many good numbers. Then, since $m$ is good, 
			$$
			2^{n}A_n<2^{n-4\nu}A_{s_1}<\cdots <2^{n-4j\nu}A_{s_j}\leq 2^{s_j-2j\nu}A_{s_j}\leq 2^{m-2j\nu} A_{m}\leq 2^{2\nu} 2^{m}A_{m}\to 0
			$$ 
			as $n \to\infty$, by \eqref{EQ-goodinfty}.
			
			Suppose now that $n<s_1$. Then either there is a good number $m$ such that $n<m\leq s_1$, or $s_1$ is bad, in which case we can find $s_2<s_1$ such that $|s_2-s_1|\leq 2\nu$ and $A_{s_1}<2^{-4\nu}A_{s_2}$ (note that the case $s_2<s_1$ is not possible, since it leads to a contradiction as above). Similarly as before, we iterate the procedure until we find a set $\{s_{j-1},s_{j-1}+1,\ldots ,s_j\}$ that contains at least one good number. Since there are infinitely many good numbers, we arrive at a finite sequence $n=s_0<s_1<\cdots <s_{j-1}<s_j$, where the numbers in the set $\{s_{0},s_0+1,\ldots s_{j-1}\}$ are bad, and there is a good number $m$ such that $s_{j-1}<m\leq s_j$. Fix $m$ to be any good number from $\{s_{j-1}+1,s_{j-1}+2,\ldots , s_j\}$. Since  $n<m$, we have
			$$
			2^{n}A_n<2^{n-4\nu}A_{s_1}<\cdots< 2^{n-4j\nu}A_{s_j}< 2^{m-4j\nu}B_m\leq 2^mA_m \to 0
			$$
			as $n\to \infty$, by \eqref{EQ-goodinfty}.
			
			Assume now there are finitely many good numbers $n$. Assume that $N\in \N$ is such that $m\leq N$ for all good numbers $m$. If $n>N$, then $n$ is a bad number, so that $A_n<2^{-4\nu}B_n$, and $B_n=A_{s_1}$ for some $s_1$ satisfying $|n-s_1|\leq 2\nu$. If $s_1<n$, one can find, in a similar way as above, a sequence $n=s_0>s_1>\cdots >s_{j-1}>s_j$, where $s_0,s_1,\ldots ,s_{j-1}$ are bad and $s_j$ is good, and moreover $n\leq s_j+2j\nu$. Since $s_j\leq N$,
			\begin{equation}
			\label{EQineqI}
			j\geq \frac{n-s_j}{2\nu}\geq \frac{n-N}{2\nu},
			\end{equation}
			and we deduce
			$$
			2^{n} A_n<2^{n-4\nu}A_{s_1}<\cdots<2^{n-4j\nu}A_{s_j}\leq 2^{s_j-2j\nu}A_{s_j}\leq 2^{N-2j\nu}\max_{0\leq k\leq N}A_k.
			$$
			The latter vanishes as $n\to \infty$, since in such a case $j\to \infty$, by \eqref{EQineqI}. 
			
			Finally, we are left to investigate the case $s_1>n$. We actually show that this case is not possible. Let $n$ be such that $A_n>0$ (if this $n$ does not exist, our assertion follows trivially). If $s_1>n$, then there is an infinite sequence of bad numbers $n=s_0<s_1<s_2<\cdots$ such that $A_{s_j}<2^{-4\nu}B_{s_j}=2^{-4\nu}A_{s_{j+1}}$ for every $j\geq 0$. Hence,
			$$
			\frac{A_{s_{j+1}}}{A_{s_j}}>2^{4\nu} \qquad \text{for all }j\geq 0,
			$$
			i.e., the sequence $A_{s_k}$ does not vanish as $k\to \infty$. This  contradicts the hypothesis of $f$ vanishing at infinity, showing the case $s_1>n$ is not possible and thus completing the proof.
		\end{proof}

	A version of Theorem~\ref{THM-main-functions} for $GMS$ can be derived easily:
		\begin{corollary}\label{CORgms}
			Let $\{a_n\}\in GMS$ be  real-valued. If the series $\sum_{n=0}^\infty a_n$	converges, then
			$$
			na_n\to 0 \qquad \text{as }n\to \infty.
			$$
		\end{corollary}
		\begin{proof}
			Let
			$$
			f(x)=a_n, \qquad x\in (n, n+1], \, n\in \N\cup\{0\}.
			$$
			It is clear that $f\in GM$ if and only if $\{a_n\}\in GMS$. Moreover, the convergence of \eqref{EQseries} is equivalent to the convergence of \eqref{EQintegrals}. Applying Theorem~\ref{THM-main-functions}, we derive that $xf(x)\to 0$ as $x\to \infty$, or in other words, $n|a_n|\to 0$ as $n\to \infty$.
		\end{proof}
	
		\begin{proof}[Proof of Corollary~\ref{CORrealtest-functions}] First we note that if $f\in GM$, then $t^\nu f(t)\in GM$ for every $\nu\in\R$.
			Integration by parts along with the condition $t^\nu f(t)\to 0$ as $t\to 0$ implies that for any $N\in \R_+$,
			$$
			\int_0^N t^{\nu-1}f(t)\, dt=\frac{1}{\nu}N^\nu f(N)-\frac{1}{\nu}\int_0^N t^\nu df(t).
			$$
			Letting $N\to \infty$ yields the desired result, where we apply Theorem~\ref{THM-main-functions} to prove the ``only if'' part.
		\end{proof}
		An analogous result to Corollary~\ref{CORrealtest-functions} holds for $GMS$. This can be easily proved by combining Corollary~\ref{CORgms} and Abel's summation formula.
		\begin{remark}
		A multidimensional version of Corollary~\ref{CORrealtest-functions} for the so-called \textit{weak monotone sequences} (which are not comparable to $GMS$) is proved in \cite{DLTZ}.
		\end{remark}

		\section{Proofs}\label{SEC4}
	
		Let us now prove Theorems~\ref{boundednesscos}~and~\ref{CORunifconv}. 	We also show that in the general case, the assertions of Corollary~\ref{CORcostrans} and Theorem~\ref{CORunifconv} are not true, or in other words, these results are sharp with respect to the general monotonicity condition. However, some of the implications of Corollary~\ref{CORcostrans} and Theorem~\ref{CORunifconv} remain true even in the general case, as we will see.
		
		We remark that the main contribution of this paper is the proof that (i) implies (ii) in both theorems, since the rest was already known or is rather trivial. In fact, that (i) implies (ii) in Theorem~\ref{boundednesscos} was known to be true under the assumption 
		\begin{equation}
		\label{EQaux1}
		t^{2\alpha+2}f(t)\to 0 \qquad \text{as }t\to \infty,
		\end{equation} 
		see \cite{mioHankel,DLT}. The fact that (i) implies (ii) in Theorem~\ref{CORunifconv} under the assumption 
		\begin{equation}
		\label{EQaux2}
		na_n\to 0 \qquad \text{as }n\to \infty
		\end{equation} 
		is included in the statement of Theorem~\ref{THMuccosine}. Theorem~\ref{THM-main-functions} and Corollary~\ref{CORgms} allow us to show that \eqref{EQaux1} and \eqref{EQaux2} are redundant if $\int_0^\infty t^{2\alpha+1}f(t)\, dt$ and $\sum_{n=0}^\infty a_n$ converge, respectively, with $f\in GM$ and $\{a_n\}\in GMS$.

		We emphasize that we deal with real-valued $GM$ functions and $GMS$, since in the non-negative case, the problems discussed are trivial.

		Before proving Theorem~\ref{boundednesscos}, we need to show that $M_{2\alpha+2}(f)$ is finite given the hypotheses of Theorem~\ref{boundednesscos}. To prove this, we use the following known estimate for $GM$ functions (see \cite{LTnachr}):
		\begin{lemma}
			\label{LEMgmest}
			Let  $f\in GM$. The estimate
			$$
			|f(t)|\lesssim \int_{t/\lambda}^{\lambda t}\frac{|f(s)|}{s}\, ds
			$$
			holds for every $t>0$.
		\end{lemma}
		
		\begin{lemma}Let  $f\in GM$ be real-valued and $\alpha\in \R$.  If $t^{2\alpha+1}f(t)\in L^1(0,1)$ and $\int_0^\infty t^{2\alpha+1}f(t)\, dt$ converges, then $M_{2\alpha+2}(f)=\sup_{t\in \R_+}t^{2\alpha+2} |f(t)|<\infty$.
		\end{lemma}
		\begin{proof}
		On the one hand, the fact that $t^{2\alpha+2}  f(t)\to 0$ as $t\to 0$ follows from $t^{2\alpha+1}f(t)\in L^1(0,1)$ and the  estimate given in Lemma~\ref{LEMgmest}. 	On the other hand, $t^{2\alpha+2} f(t)\to 0$ as $t\to \infty$ follows from the convergence of $\int_0^\infty t^{2\alpha+1}f(t)\, dt$ and Theorem~\ref{THM-main-functions} (recall that $t^{2\alpha+1}f(t)\in GM$ provided that $f\in GM$). Finally, since $f$ is locally of bounded variation, $t^{2\alpha+2}|f(t)|$ is bounded on any compact set, which yields the desired result.
		\end{proof}
		
		\begin{proof}[Proof of Theorem~\ref{boundednesscos}]
			It was proved in \cite{mioHankel} that (i) and (ii) are equivalent provided that Theorem~\ref{THM-main-functions} is true. 
			
			We now prove that (i) and (iii) are equivalent. That (iii) implies (i) is clear even in the general case, since if $H_\alpha f(u)$ is bounded, then $H_\alpha f(0)=\int_0^\infty t^{2\alpha+1}f(t)\, dt$ converges. So we are left to prove that if $f\in GM$, then (i) implies (iii). It suffices to prove estimate \eqref{EQ-cossup}, and the claim follows by letting $N\to \infty$.
			
			Let $u>0$ (the case $u=0$ is trivial, since $j_\alpha(0)=1$). First of all, we write
			$$
			\bigg|\int_0^N t^{2\alpha+1}f(t)j_\alpha(ut)\, dt\bigg| \leq \bigg|\int_0^N t^{2\alpha+1}f(t)\, dt\bigg|+\bigg|\int_{0}^{N} t^{2\alpha+1}f(t)(1-j_\alpha(ut))\, dt \bigg|.
			$$
			Applying integration by parts to the second integral on the right hand side of the latter together with \eqref{EQ-derivative-bessel} and the fact that $|j_\alpha(t)|\leq 1$ for all $t\geq 0$, we get
			\begin{align*}
			\bigg|\int_0^N t^{2\alpha+1}f(t)(1-j_\alpha(ut))\,dt\bigg| &\leq  \bigg| \frac{t^{2\alpha+2}}{2\alpha+2} (1-j_{\alpha+1}(ut))f(t) \bigg|_{t=0}^N \bigg| \\
			&\phantom{=} +\bigg| \int_0^N \frac{t^{2\alpha+2}}{2\alpha+2}(1-j_{\alpha+1}(ut))df(t)\bigg|\\
			&\leq \frac{1}{\alpha+1}N^{2\alpha+2}|f(N)| \\
			&\phantom{=}+ \bigg|\int_{0}^N \frac{t^{2\alpha+2}}{2\alpha+2}(1-j_{\alpha+1}(ut))df(t) \bigg|,
			\end{align*}
			where in the last inequality we have used that $t^{2\alpha+2}f(t)\to 0$ as $t\to 0$, which follows from $t^{2\alpha+1}f(t)\in L^1(0,1)$, $f\in GM$, and Lemma~\ref{LEMgmest}. 		
			
			Assume now that $u\leq 1/N$. By Lemma~\ref{LEM-Besselest}, we have $j_{\alpha+1}(ut)\geq 1-(ut)^2/(4(\alpha+2))$ for $t\leq N$, and therefore
			$$
			\bigg|\int_{0}^N \frac{t^{2\alpha+2}}{2\alpha+2}(1-j_{\alpha+1}(ut))df(t) \bigg|\leq \frac{1}{B_\alpha N^2}\int_0^N t^{2\alpha+4}|df(t)|,
			$$
			where $B_\alpha=4(\alpha+2)(2\alpha+2)$.

			Let $n_0=\min\{k\in \mathbb{Z}: 2^k\geq N\}$. Using the $GM$ condition, we have
			\begin{align}
			\frac{1}{B_\alpha N^2}\int_0^N t^{2\alpha+4}|df(t)|& \leq \frac{2^{2(1-n_0)}}{B_\alpha}\sum_{k=-\infty}^{n_0}2^{(2\alpha+4)k}\int_{2^{k-1}}^{2^k}|df(t)|\nonumber \\
			&\leq \frac{C2^{2(1-n_0)}}{B_\alpha}\sum_{k=-\infty}^{n_0}2^{(2\alpha+4)k}\int_{2^{k-1}/\lambda}^{\lambda 2^{k-1}}\frac{|f(t)|}{t}\, dt\nonumber \\
			&\leq \frac{C2^{2(1-n_0)}(2\lambda)^{(2\alpha+4)}}{B_\alpha} \sum_{k=-\infty}^{n_0}\int_{2^{k-1}/\lambda}^{\lambda 2^{k-1}} t^{2\alpha+3}|f(t)|\, dt\nonumber \\ 
			&\leq \frac{C2^{2(1-n_0)}(2\lambda)^{(2\alpha+5)}}{2B_\alpha} M_{2\alpha+2}(f) \int_{0}^{\lambda 2^{n_0-1}} t\, dt \nonumber\\
			&= \frac{C(2\lambda)^{(2\alpha+7)}}{16B_\alpha}M_{2\alpha+2}(f).\label{EQestaux}
			\end{align}

			Assume now $u>1/N$. Then
			$$
			\bigg|\int_{0}^N \frac{t^{2\alpha+2}}{2\alpha+2}(1-j_{\alpha+1}(ut))df(t) \bigg|= \bigg|\bigg( \int_0^{1/u}+\int_{1/u}^N\bigg) \frac{t^{2\alpha+2}}{2\alpha+2}(1-j_{\alpha+1}(ut))df(t) \bigg|. 
			$$
			On the one hand, using the estimate \eqref{EQestaux} we get
			$$
			\bigg|\int_0^{1/u}\frac{t^{2\alpha+2}}{2\alpha+2}(1-j_{\alpha+1}(ut))df(t) \bigg|\leq \frac{C(2\lambda)^{(2\alpha+7)}}{16B_\alpha}M_{2\alpha+2}(f).
			$$
			On the other hand,
			\begin{align*}
				\bigg|\int_{1/u}^N   \frac{t^{2\alpha+2}}{2\alpha+2}(1-j_{\alpha+1}(ut))df(t) \bigg|&\leq \bigg|\int_{1/u}^N \frac{t^{2\alpha+2}}{2\alpha+2}df(t)\bigg|\\
				&\phantom{=}+\int_{1/u}^N \bigg| \frac{t^{2\alpha+2}}{2\alpha+2} j_{\alpha+1}(ut) df(t) \bigg|.
			\end{align*}
			The first integral on the right hand side of the latter is less than or equal to 
			$$
			\sup_{0\leq a<b\leq \infty}\bigg|\int_a^b \frac{t^{2\alpha+2}}{2\alpha+2}df(t)\bigg|,
			$$
			which is finite due to the convergence of $\int_0^\infty t^{2\alpha+2}df(t)$, by Corollary~\ref{CORrealtest-functions}. Finally, let $n_1=\max\{k\in \mathbb{Z}: 2^k\leq 1/u\}$. Using the $GM$ condition and estimate \eqref{EQ-besselestinfity}, we derive
			
			\begin{align*}
			 \bigg|\int_{1/u}^N \frac{t^{2\alpha+2}}{2\alpha+2}j_{\alpha+1}(ut)df(t)\bigg|&\leq \frac{S_{\alpha+1}}{u^{\alpha+3/2}}\int_{1/u}^\infty \frac{t^{\alpha+1/2}}{2\alpha+2}|df(t)|\\
			& \leq \frac{2^{(n_1+1)(\alpha+3/2)}S_{\alpha+1}}{2\alpha+2}\sum_{k=n_1}^\infty 2^{(k+1)(\alpha+1/2)}\int_{2^k}^{2^{k+1}}|df(t)|\\
			&\leq \frac{C2^{(n_1+1)(\alpha+3/2)}S_{\alpha+1}}{2\alpha+2}\sum_{k=n_1}^\infty  2^{(k+1)(\alpha+1/2)} \int_{2^k/\lambda}^{\lambda 2^k}\frac{|f(t)|}{t}\, dt.
			\end{align*}
		Since 
		\begin{align*}
		\sum_{k=n_1}^\infty  2^{(k+1)(\alpha+1/2)} \int_{2^k/\lambda}^{\lambda 2^k}\frac{|f(t)|}{t}\, dt & \leq \lambda(2\lambda)^{\alpha+1/2} \int_{2^{n_1}/\lambda}^\infty t^{\alpha-1/2}|f(t)|\, dt\\
		&\leq \lambda(2\lambda)^{\alpha+1/2}M_{2\alpha+2}(f)\int_{2^{n_1}/\lambda}^\infty t^{-\alpha-5/2}\, dt\\
		&= \frac{2^{\alpha+1/2}\lambda^{2\alpha+3}}{\alpha+3/2}2^{-n_1(\alpha+3/2)}M_{2\alpha+2}(f),
		\end{align*}
		we conclude that
		$$
		\bigg|\int_{1/u}^N \frac{t^{2\alpha+2}}{2\alpha+2}j_{\alpha+1}(ut)df(t)\bigg|\leq \frac{C\lambda(2\lambda)^{2\alpha+2}S_\alpha}{(\alpha+3/2)(2\alpha+2)}M_{2\alpha+2}(f).
		$$
		Collecting the above estimates, we arrive at \eqref{EQ-cossup}.
		\end{proof}
		\begin{remark}
			Although the proof of Theorem~\ref{boundednesscos} can be done in a much simpler way if we disregard the estimate \eqref{EQ-cossup} (as in Theorem~\ref{CORunifconv}), we prefer to include it to show its dependence on the parameter $\alpha$.
		\end{remark}

In order to prove Theorem~\ref{CORunifconv}, we need the following observation.
\begin{remark}
	\label{REM-unifconvseries}
	In the proof of Theorem~\ref{THMuccosine}, the hypothesis $n|a_n|\to 0$ is only used to prove the ``if'' part, whilst the ``only if'' part only requires the convergence of \eqref{EQseries}.
\end{remark}		

\begin{proof}[Proof of Theorem~\ref{CORunifconv}]
	The convergence of \eqref{EQseries} implies $n|a_n|\to 0$ as $n\to \infty$, by Corollary~\ref{CORgms}. Therefore, according to Remark~\ref{REM-unifconvseries}, the equivalence of (i) and (ii) follows by applying Theorem~\ref{THMuccosine}, since we have shown the hypothesis $n|a_n|\to 0$ is redundant in order to prove the ``if'' part. Also, as we observed in the Introduction, the uniform convergence of the partial sums $\sum_{n=0}^N a_n\cos nx$ implies the boundedness of the limit function $\sum_{n=0}^\infty a_n\cos nx$, i.e., (ii) implies (iii). Finally, (iii) trivially implies (i) by choosing $x=0$.
\end{proof}

\noindent \textbf{Sharpness}.
To conclude, let us discuss what implications of Corollary~\ref{CORcostrans} and Theorem~\ref{CORunifconv} hold in the general case and which ones do not. We start with the cosine transform. For a given $f\in L^1(0,1)$, since $\widehat{f}_{\cos}(0)=\int_0^\infty f(t)\, dt$, the convergence of the integral \eqref{EQintegrals} is necessary for the (pointwise, and therefore also uniform) convergence and boundedness of $\widehat{f}_{\cos}$, but not sufficient. Indeed, consider the integral (cf. \cite[pp. 7--8]{EMOTtables})
$$
\widehat{f}_{\cos}(u)=\int_0^\infty t^{-1/2} \cos t\cos ut\, dt=\frac{\sqrt{\pi}}{2\sqrt{2}}\bigg(\frac{1}{\sqrt{u+1}}+\frac{1}{\sqrt{u-1}}\bigg), \qquad u>0.
$$
It is clear that, for $0<u\leq 1$, $\widehat{f}_{\cos}(u)$ does not even converge, and moreover it tends to infinity as $u\to 1^+$, although, as is well known, the integral $\int_0^\infty t^{-1/2}\cos t\, dt$ converges (see Fresnel integrals, \cite[pp. 300--301]{AS}). We also show that uniform convergence implies the boundedness of the limit function, but not vice-versa. Certainly, if $\widehat{f}_{\cos}$ converges uniformly, for a fixed $\varepsilon >0$ we can find $N\in \R_+$ such that
$$
\bigg|\int_{M_1}^{M_2}f(t)\cos ut\, dt\bigg|<\varepsilon,\qquad \text{if }N\leq M_1<M_2,
$$
and hence,
$$
\bigg|\int_0^\infty f(t)\cos ut\, dt\bigg|\leq \int_0^N|f(t)|\, dt+\varepsilon<\infty,
$$
since $f\in L^1(0,1)$ is locally integrable on $(0,\infty)$. To see the contrary is not true, we just take 
$$
f(t)=\begin{cases}
e^{-t},&\text{if }0\leq t\leq 1,\\
0, &\text{if }t>1.
\end{cases}
$$
It can be easily shown, integrating by parts twice, that
$$
g(u):=\widehat{f}_{\cos}(u)=\frac{1-e^{-1}\cos u+e^{-1}u\sin u}{1+u^2}.
$$
Since $f$ is piecewise smooth and integrable, we have that $\widehat{g}_{\cos}(t)=f(t)$ almost everywhere in $t\in \R_+$. In particular, $\widehat{g}_{\cos}$ is bounded and converges  to a discontinuous function, thus the convergence cannot be uniform, since the partial integrals $\int_0^N g(u)\cos ut\, du$ are continuous.

In the case of cosine series, the situation is analogous. The convergence of $\sum_{n=0}^\infty a_n$ is necessary but not sufficient to guarantee the uniform convergence and boundedness of $\sum_{n=0}^\infty a_n\cos nx$. Indeed, the necessity part is trivial, as for the sufficiency part let us consider $a_n=\frac{\cos n}{n}$, $n\geq 1$. By the well known Dirichlet test for series convergence, it follows that $\sum_{n=1}^\infty \frac{\cos n}{n}$ converges. However, the cosine series
$$
\sum_{n=1}^\infty \frac{\cos n}{n}\cos nx
$$
diverges at $x=1$, thus it is not bounded neither its partial sums converge uniformly. Indeed, to see this we first observe that
$$
\sum_{n=1}^N \cos^2 n =\sum_{n=1}^N \bigg( \frac{e^{in}+e^{-in}}{2}\bigg)^2 =\frac{N}{2}+ \sum_{n=1}^N \frac{e^{2in}+e^{-2in}}{4} = \frac{N}{2}+\frac{D_N(2)-1}{4},
$$
where $D_N(x)$ denotes the Dirichlet kernel (see \cite{zygmund})
$$
D_N(x)=\sum_{n=-N}^N e^{ikx} = \frac{\sin (N+1/2)x}{\sin x/2}.
$$
Applying Abel's transformation, we obtain
$$
\sum_{n=1}^\infty \frac{\cos^2 n}{n} = \frac{1}{N}\sum_{n=1}^N \cos^2 n  -\sum_{k=1}^{N-1}\frac{1}{k}\sum_{n=1}^k  \cos^2 n.
$$
Joining the above equalities, it readily follows that the series $\sum_{n=1}^\infty n^{-1}\cos^2 n$ diverges.

The fact that the uniform convergence of partial sums of a cosine series imply its boundedness may be proved in the same way as the analogue for cosine integrals we showed above. One may also apply the following argument: since the partial sums are continuous, the limit function $g(x)=\sum a_n\cos nx$ is continuous, because continuity is preserved by the uniform convergence. The function $g(x)$ is periodic, thus its boundedness follows, since a periodic continuous function must be bounded. Finally, we prove the boundedness of a cosine series does not imply its partial sums converge uniformly. Indeed, consider the function
$$
g(x)=\begin{cases}
1, &\text{if } x\in [\pi/2,3\pi/2],\\
0, &\text{otherwise}.
\end{cases}
$$
If we extend $g$ periodically to be defined on $\R$ is not hard to show that the Fourier series of $g$, say $G$, is actually a cosine series, since $g$ is even, and 
$$
G(x)=\sum_{n=1}^\infty \frac{(-1)^n}{2n-1}\cos \big((2n-1)x\big).
$$
Since $g(x)\in L^1(0,2\pi)$ and it is piecewise smooth, $G(x)=g(x)$ almost everywhere in $x\in [0,2\pi)$. Since the continuous partial sums 
$$\sum_{n=1}^N \frac{(-1)^n}{2n-1}\cos \big((2n-1)x\big)$$
converge almost everywhere to a discontinuous function as $N\to \infty$, the convergence cannot be uniform.

\end{document}